\documentclass[12pt]{article}
\input epsf.tex
\usepackage{amsthm, amssymb, amsmath}
\newtheorem{Theorem}{Theorem}
\newtheorem{lemma}{Lemma}

\begin{document}
\title{\bf{Relationships between the delayed chemostat and the delayed logistic equation}}
\author{Torsten Lindstr\"{o}m \\
Department of Mathematics \\
Linnaeus University \\
SE-35195 V\"{a}xj\"{o}, SWEDEN}
\date{}
\maketitle

\begin{abstract}
The logistic equation is often considered as a simplification of the chemostat (Kooi, Boer, Kooijman (1998)\nocite{Kooi.BoMB:60}). However, the global qualitative properties of the delayed single species chemostat are completely known, see e. g. Smith (2011)\nocite{smith_delays}. Such properties still remain open for the delayed logistic equation, see e. g. B{\'{a}}nheley, Czendes, Kristin, and Neymaier (2014)\nocite{Banhelyi.SIAMJADS:29} despite that they have been announced a long time ago (Wright (1955)\nocite{wright.JfRAM:194}). We may therefore ask whether the logistic equation really is a simplification and what information about the chemostat actually is contained. We discuss the links between these equations and conclude that they are less clear in the delayed case than they are in the non-delayed case. We keep the presentation as elementary as possible.
\end{abstract}

\section{Introduction}

In this paper we discuss the relation between the delayed chemostat
\begin{eqnarray}
S^\prime(T)&=&CD-DS-\frac{AS(T)}{1+ABS(T)}X(T),\nonumber\\
X^\prime(T)&=&Me^{-DR}\frac{AS(T-R)}{1+ABS(T-R)}X(T-R)-DX(T).\label{delay_chemostat}
\end{eqnarray}
and various forms of the delayed logistic equation that are going to be introduced later. Here, $S(T)$ is the substrate concentration at time $T$, $X(T)$ is the concentration of some organism at time $T$ having the
substrate $S$ as its limiting resource. The parameters $C>0$, $D>0$, $A>0$, $B>0$, $M>0$, $R>0$ stand for concentration, dilution rate, search rate for the organism, handling time for the organism (cf. Holling (1959)\nocite{Holling.CanEnt:91}), conversion factor for the organism, and time delay, respectively. The time delay could be included in many ways in the chemostat model, but we think that the above way is the most natural way to include it; in this way it takes into account that the organism cannot convert consumption of the substrate into own biomass instantaneously. Note the factor $e^{-DR}$, it corresponds to the fact that a part of the individuals that consume nutrient is washed out before they are able to reproduce.

The model above is a phenomenological model for a number of ecological phenomena, cf. Smith and Waltman (1995)\nocite{smithchemo}. For instance, a lake ecology may have rivers transferring a limiting resource to the lake and rivers diluting this resource.

We start our analysis by re-parameterizing our model into a dimensionless form by the changes $t=DT$, $r=DR$, $s=S/C$, $x=X/C$, $a=AC/D$, $b=ABC$, and $m=M$. For simplicity, we introduce the function $f$ defined by
\begin{displaymath}
f(s)=\frac{as}{1+bs},
\end{displaymath}
too. We note that $f$ is strictly increasing with $f(0)=0$. We arrive in the dimensionless form
\begin{eqnarray}
\dot{s}(t)&=&1-s(t)-f(s(t))x(t),\nonumber\\
\dot{x}(t)&=&me^{-r}f(s(t-r))x(t-r)-x(t)\label{delay_chemostat_dim_less}
\end{eqnarray}
for which we impose the initial conditions $s(t)\geq 0$, $-r\leq t\leq 0$ and $x(t)\geq 0$, $-r\leq t<0$, $x(0)>0$. A number of important results regarding the model (\ref{delay_chemostat_dim_less}) has been proved and can be found in Smith (2011)\nocite{smith_delays}. They are that the solution space is $C([-r,0],{\bf{R}}^2)$, solutions remain non-negative (Theorem 3.4), and bounded (Section 5.6). Either one or two equilibria, the washout state and the survival state exist, the washout state is globally attracting if it is the only equilibrium and so is the survival state it it exists (Theorem 5.16). The first proof of the global stability result can be found in Ellermeyer (1994)\nocite{SIAMJAM.Ellermeyer}. Thus, all global qualitative dynamical properties of the single species chemostat model are known.

\section{Methods}

We briefly go through the main methods that are going to be used in this paper. Most of this material can be found in the monographs of Hale and Lunel (1993)\nocite{delays}, Smith (1995)\nocite{Smith.monotone}, and Smith (2011)\nocite{smith_delays}. It turns out that (\ref{delay_chemostat_dim_less}) can be reduced to a scalar equation, so we work throughout the paper with scalar autonomous delay differential equations of the form
\begin{equation}
\dot{x}(t)=f(x(t),x(t-r))
\label{general_scalar_autonomous}
\end{equation}
satisfying an initial condition $x(t)=\phi(t)$, $-r\leq t\leq 0$. The solution of (\ref{general_scalar_autonomous}) starting at the initial condition $\phi\in C([-r,0],{\bf{R}})$ is denoted by $x(t,\phi)$. We denote the right hand side by $f(x,y)$ when we want to refer to the different variables or take partial derivatives. Existence and uniqueness of the solutions of such equations are granted by the method of steps if $f$ and $f_x$ are continuous on ${\bf{R}}^2$ and $\phi$ is continuous on $-r\leq t\leq 0$ since the equation becomes an ODE after substituting the continuous initial condition into the equation. Hence, standard ODE theory applies and grants existence and uniqueness of solutions of equations of form (\ref{general_scalar_autonomous}). It follows that the natural state space of the equation (\ref{general_scalar_autonomous}) is the continuous functions on the interval $-r\leq t\leq 0$, or $C([-r,0],{\bf{R}})$.

The equations that we are going to consider arise in population dynamics and we are going to use a theorem granting that non-negative solutions remain non-negative. Such theorems require that solutions exist and are unique, but additional conditions are required. This additional condition is given by $f(0,y)\geq 0,\:\forall t\in{\bf{R}},\forall y\in{\bf{R}}_+$ and we note that this condition can be expressed in terms of the functions involved in the differential equation directly. Therefore, direct substitution in the involved expressions can be used to check these things. Also this theorem is directly related to geometric ODE theory, it basically formulates conditions for when the vector field is either directed into the positive cone or along the boundaries of it. No problems arise when it is directed into the positive cone and in the latter case uniqueness of solutions grants that no solutions escape from the positive cone.

Studies of nonlinear ordinary differential equations usually start from analysis of equilibria. All equilibria can be found by neglecting delays and therefore they coincide with the equilibria of the corresponding ODE
\begin{equation}
\dot{x}(t)=f(x(t),x(t)).
\label{general_scalar_ODE}
\end{equation}
We claim that we have found one equilibrium by solving $f(x,x)=0$, say $x(t)=\bar{x}$. The linearization of (\ref{general_scalar_autonomous}) around the equilibrium $\bar{x}$ is then given by
\begin{equation}
\dot{\xi}(t)=f_x(\bar{x},\bar{x})\xi(t)+f_y(\bar{x},\bar{x})\xi(t-r)=a\xi(t)+b\xi(t-r).
\label{linearization}
\end{equation}
We shall as usual look for exponential solutions around the zero solution of this equation and these are characterized by the characteristic equation
\begin{displaymath}
\lambda=a+be^{-\lambda r}.
\label{char_eq}
\end{displaymath}
The solutions of (\ref{char_eq}) classifies the solutions of (\ref{linearization}) into four cases according to the following theorem (Theorem 4.7 in Smith (2011)\nocite{smith_delays}):
\begin{Theorem}
The following hold for \em (\ref{linearization}). \em
\begin{itemize}
\item[\em(A)\em] If $a+b>0$, then $\xi=0$ is unstable
\item[\em(B)\em] If $a+b<0$, $b\geq a$, then $\xi=0$ is asymptotically stable regardless of the magnitude of the delay.
\item[\em(C)\em] If $a+b<0$, $b<a$, then there exist threshold value for the delay $r_\ast$ (which can be explicitly computed) such that $\xi=0$ is asymptotically stable for $0<r<r_\ast$ and unstable for $r>r_\ast$.
\item[\em(D)\em] If $a+b=0$, then $\lambda=0$ is a root of \em (\ref{char_eq}).
\end{itemize}
\label{classification_theorem}
\end{Theorem}
In case (D) the linearization (\ref{linearization}) does not provide the qualitative information requested regarding solutions of the original nonlinear equation (\ref{general_scalar_autonomous}) in the vicinity of its equilibrium $\bar{x}$. However, center manifold theory (Guckheimer and Holmes (1983)\nocite{guck}) exist in the delay differential equation case, too, see Diekmann, van Gils, Verduyn Lunel, and Walther (1995)\nocite{Diekmann.Delay.1995}.

In ODE theory there exists a standard method for constructing Lyapunov functions for linear equations when all characteristic roots have negative real part, see Jordan and Smith (1990)\nocite{jordan}. There are extensions of this method to delay equations, see Hale and Lunel (1993)\nocite{delays}, but in most cases explicit constructions can be very complicated already in simple cases and most results exist just in the case (B) above. To get Lyapunov functions that are valid for the nonlinear equations like (\ref{general_scalar_autonomous}), such results still usually need to be modified.

The method of Lyapunov functions and LaSalle's (1960)\nocite{LaSalle.IRETCT:7} invariance principle is always the most desirable when proving global stability. However, more straightforward methods might exist. We start by mentioning the following simple comparison theorem (Smith (2011)\nocite{smith_delays}, Theorem 3.6). Essentially, it states that if a solution start in front of (behind) another and grows faster (slower), it will stay in front of (behind) the other.
\begin{Theorem}
Consider \em (\ref{general_scalar_autonomous}) \em and suppose that $f$ and $f_x$  are continuous on ${\bf{R}}^3$ and that the initial condition $\phi$ is continuous on $[-r,0]$. Let $x(t)$ be a solution of \em (\ref{general_scalar_autonomous}) \em with $x(t)=\phi(t)$ for $[-r,0]$ on some interval $[0,\hat{s}]$, $\hat{s}>0$. Let $\tilde{x}(t)$ be continuous on $[-r,0]\cup [0,\hat{s}]$, differentiable on $[0,\hat{s}]$ and satisfy
\begin{eqnarray*}
\dot{\tilde{x}}(t)&\geq (\leq)&f(\tilde{x}(t),\tilde{x}(t-r)), \: t\in [0,\hat{s}],\\
\dot{\tilde{x}}(t)&\geq (\leq)&\phi(t), -r\leq t\leq 0.
\end{eqnarray*}
Then $\tilde{x}(t)\geq (\leq) x(t)$.
\end{Theorem}
We continue formulating conditions that preserve the order of two solutions as follows. It is Theorem 5.9 in Smith (2011)\nocite{smith_delays}.
\begin{Theorem}
Consider \em (\ref{general_scalar_autonomous}). \em Assume that $f:{\bf{R}}^2\rightarrow {\bf{R}}$, $f_x$, $f_y$ are continuous, $f_y(x,y)\geq 0$, and that $x_1(t)$ and $x_2(t)$ are two solutions of \em (\ref{general_scalar_autonomous}) \em defined on $[-r,\hat{s}]$, for some $\hat{s}$. If $x_1(t)\leq x_2(t)$ for $t\in [-r,0]$, then $x_1(t)\leq x_2(t)$ for $t\in [-r,\hat{s}]$.
\label{monotone_Th}
\end{Theorem}
This theorem is a direct consequence of the previous theorem and makes it possible to translate many of the properties of the equilibria of the scalar ordinary differential equation (\ref{general_scalar_ODE}) to (\ref{general_scalar_autonomous}). The following theorem contains the parts of Corollary 5.11 in Smith (2011)\nocite{smith_delays} that we shall need.
\begin{Theorem}
Let $\bar{x}\in(a,b)$ be an equilibrium of \em (\ref{general_scalar_ODE}), \em and hence of \em (\ref{general_scalar_autonomous}). \em Suppose that the sign-condition
\begin{displaymath}
(x-\bar{x})f(x,x)<0,\:\: x\in [a,b],\:\: x\neq \bar{x}
\end{displaymath}
holds. If $\phi(s)\in[a,b]$, $s\in[-r,0]$, then
\begin{displaymath}
\lim_{t\rightarrow\infty}=\bar{x}.
\end{displaymath}
\label{monotone_scalar_eq_thm}
\end{Theorem}

\section{A delayed hyperbolic model}

A relation between the single species chemostat model and the logistic model exists, see e. g. Kooi, Boer, Kooijman (1998)\nocite{Kooi.BoMB:60} and this relation can be made visible also in the delayed case. More precisely, the logistic equation can be considered as an approximation of the mass-balance equations used in the chemostat. Let us consider the functional
\begin{equation}
V(t)=x(t)+me^{-r}s(t-r)-me^{-r}.
\label{lyap_functional}
\end{equation}
Along the solutions of (\ref{delay_chemostat_dim_less}), we identify that equation for the total derivative is the ordinary differential equation ${\dot{V}}(t)=-V(t)$ and thus, $\lim_{t\rightarrow\infty}V(t)=0$. We ask what models we could obtain by considering the delayed chemostat in the hyperplane $V=0$ in $C([-r,0],{\bf{R}}^2)$. It is not evident that such a procedure will preserve the dynamics even in the finite dimensional case (Thieme (1992)\nocite{JoMB.Thieme:30}) but we try this substitution and get
\begin{equation}
\dot{x}(t)=me^{-r}f\left(\frac{me^{-r}-x(t)}{me^{-r}}\right)x(t-r)-x(t).
\label{hyperbolic}
\end{equation}
This is a single hyperbolic delayed equation. Our program is for the moment to verify that this equation has the same qualitative dynamical properties as the single species chemostat model. Indeed, we prove the following sequence of results.
\begin{lemma}
Solutions of \em (\ref{hyperbolic}) \em remain non-negative and bounded above by $x(t)<me^{-r}$.
\label{pos_bounded}
\end{lemma}
\begin{proof}
We start be proving that the solutions remain non-negative. This follows from the fact that the estimate
\begin{displaymath}
me^{-r}f(1)x(t-r)\geq 0
\end{displaymath}
holds for all positive $x(t-r)$. Theorem 3.4 in Smith (2011)\nocite{smith_delays} gives the result. Comparison with
\begin{displaymath}
\dot{x}(t)\leq -x(t)
\end{displaymath}
for $x(t)\geq me^{-r}$ shows that large positive solutions decay at least exponentially since $f(s)<0$ for $s<0$. It follows that positive solutions are bounded above by $x(t)<me^{-r}$.
\end{proof}
\begin{lemma}
The equation \em(\ref{hyperbolic}) \em has the washout state $x=0$ as its unique non-negative equilibrium if $f(1)<\frac{e^{r}}{m}$. This equilibrium is then locally stable. If $f(1)>\frac{e^{r}}{m}$, then \em(\ref{hyperbolic}) \em has two non-negative equilibria, the washout state that is locally unstable and the survival state that is locally stable.
\end{lemma}

\begin{proof}
By neglecting delays, we get that the washout state is an equilibrium. For the survival state we solve
\begin{displaymath}
f\left(\frac{me^{-r}-x(t)}{me^{-r}}\right)=\frac{e^r}{m}
\end{displaymath}
and conclude that this equation has a unique positive solution $x(t)=\bar{x}$ if and only if $f(1)>\frac{e^{r}}{m}$. In order to investigate the local stability of these solutions, we compute the characteristic equation of a generic equilibrium of the equation (\ref{hyperbolic}). It takes the form
\begin{displaymath}
\lambda+1+xf^\prime(s)-e^{-\lambda r}me^{-r}f(s)=0
\end{displaymath}
that turns out to be the second factor of the characteristic equation of the delayed single species chemostat (\ref{delay_chemostat_dim_less}), see Smith (2011, p55)\nocite{smith_delays}. The factorization of this equation given there is therefore, not a coincidence. It is closely related to motion in the hyperplane $V=0$ defined by (\ref{lyap_functional}) and motion towards it. Hence, we have proved that the local stability properties of the delayed single species chemostat (\ref{delay_chemostat_dim_less}) and the hyperbolic model (\ref{hyperbolic}) are the same.
\end{proof}

We go on formulating the global properties of (\ref{hyperbolic}) and remark that we use the monotone dynamics of it in order to arrive in the conclusions. This is a simplification in comparison to proving the global stability properties of the chemostat equations (\ref{delay_chemostat_dim_less}) directly, since the first equation does not satisfy the quasimonotone condition in Smith (2005)\nocite{Smith.monotone}.
\begin{Theorem}
The following statements hold for the solutions of \em (\ref{hyperbolic})\em.
\begin{itemize}
\item[\em (i)\em] If $f(1)<\frac{e^{r}}{m}$, then the washout equilibrium is the unique non-negative equilibrium of \em (\ref{hyperbolic}) \em and it attracts all non-negative solutions.
\item[\em (ii)\em] If $f(1)>\frac{e^{r}}{m}$, then the survival equilibrium exists for \em (\ref{hyperbolic}) \em and it attracts all positive solutions.
\end{itemize}
\end{Theorem}

\begin{proof}
Consider (\ref{hyperbolic}). By Lemma \ref{pos_bounded}, we have $0\leq x(t)< me^{-r}$. With aid of the differential inequality
\begin{displaymath}
\dot{x}(t)\leq me^{-r}f(1)x(t-1)-x(t)
\end{displaymath}
we notice that no solutions can grow faster than the solutions of the linear delay differential equation
\begin{equation}
\dot{z}(t)=me^{-r}f(1)z(t-1)-z(t)
\label{comparison_eq}
\end{equation}
which has the characteristic equation $\lambda=-1+me^{-r}f(1)e^{-\lambda}$. Now
\begin{equation}
me^{-r}f(1)<1
\label{wash_out_stable}
\end{equation}
implies that the zero solution of (\ref{comparison_eq}) is asymptotically stable and falls into category (B) of Theorem \ref{classification_theorem}. Therefore, all solutions of (\ref{hyperbolic}) converge to zero if (\ref{wash_out_stable}) since they are bounded from above by solutions that decay to the zero solution.

Now assume $me^{-r}f(1)>1$ and remember that $x(0)>0$. The differential inequality $\dot{x}(t)>-x(t)$ gives $x(t)>0$. We can therefore, work with initial conditions satisfying $x(t)\geq A$, for some $A>0$ and $-r\leq t\leq 0$. Consider (\ref{hyperbolic}). It possesses monotone dynamics (or order-preserving dynamics, see Smith (1995)\nocite{Smith.monotone}) since $0\leq x(t)\leq me^{-r}$ gives
\begin{displaymath}
me^{-r}f\left(\frac{me^{-r}-x(t)}{me^{-r}}\right)\geq 0
\end{displaymath}
so that Theorem \ref{monotone_Th} can be applied. Consider the non-negative equilibria of (\ref{hyperbolic}) that are $0$ and $\bar{x}$ when $me^{-r}f(1)>1$. Both equilibria are shared with the ODE
\begin{displaymath}
\dot{x}(t)=\left(me^{-r}f\left(\frac{me^{-r}-x(t)}{me^{-r}}\right)-1\right)x(t).
\end{displaymath}
For the above ODE, the survival state is globally asymptotically stable on $x(t)\in(0,me^{-r})$. Theorem \ref{monotone_scalar_eq_thm} gives that $x(t)=\bar{x}$ attracts all solutions of (\ref{hyperbolic}) with initial conditions satisfying $A\leq x(t)\leq B$ with $A>0$ and $B<me^{-r}$. In the beginning of this proof and in the proof of Lemma \ref{pos_bounded}
we have justified that we can work with such initial conditions provided $x(0)>0$ and $x(t)\geq 0$ for $-r\leq t <0$.\end{proof}
Thus, the delayed hyperbolic model (\ref{hyperbolic}) preserves the dynamics of the delayed single species chemostat model (\ref{delay_chemostat_dim_less}). We conclude that $b=0$ gives $f(s)=as$ and therefore the delayed logistic model
\begin{equation}
\dot{x}(t)=ame^{-r}x(t-r)-x(t)-ax(t)x(t-r).
\label{chemo_logistic}
\end{equation}
shares the same dynamical properties, too. In particular, it has monotone dynamics when $0\leq x(t)\leq me^{-r}$. If we let the delay parameter tend to zero, this equation becomes the ordinary differential equation
\begin{displaymath}
\dot{x}(t)=(am-1)x(t)-ax^2(t)=(am-1)x(t)\left(1-\frac{x(t)}{\frac{am-1}{a}}\right).
\end{displaymath}
We can identify the growth rate and carrying capacity parameters as $am-1$ and $(am-1)/a$, respectively. The washout equilibrium is still given by $x(t)=0$ and the survival equilibrium is located at the carrying capacity $x(t)=(am-1)/a$.

\section{The logistic equation}

Delayed logistic equations are not always derived in a way that preserves the dynamical properties of the underlying chemostat equations. Consider, for instance, Hutchinson's (1948)\nocite{Hutchinson.Ann_NYSci:50} delayed logistic model
\begin{equation}
\dot{x}(t)=(am-1)x(t)\left(1-\frac{x(t-r)}{\frac{am-1}{a}}\right).
\label{Hutchinson_logistic}
\end{equation}
We remind the reader that the resource dynamics was not modeled explicitly in Hutchinson's paper and that the problems mentioned by Kooi, Boer, Kooijman (1998)\nocite{Kooi.BoMB:60} exist. This equation can be made comparable to (\ref{chemo_logistic}) by arranging it as
\begin{displaymath}
\dot{x}(t)=am\:\:\cdot\:\:x(t\:\:\:\:)-x(t)-ax(t)x(t-r).
\end{displaymath}
We note that the above equation has preserved the equilibria at $x(t)=0$ and $x(t)=(am-1)/a$. However, it does not have monotone (order-preserving) solutions, since 
\begin{displaymath}
-ax(t)\leq 0.
\end{displaymath}
We start by linearizing around a generic equilibrium $x$ and get the characteristic equation
\begin{displaymath}
\lambda=(am-1-ax)-e^{-r\lambda}ax.
\end{displaymath}
For $x(t)=0$, we have $\lambda=am-1$ and for $x(t)=(am-1)/a$, we have
\begin{displaymath}
\lambda=-e^{-r\lambda}(am-1)
\end{displaymath}
which according to Theorem \ref{classification_theorem} belong to the category (C) above: There exists a Hopf-bifurcation value for the magnitude of the delay, so that the equilibrium is stable if the delay is below that threshold and unstable otherwise. We have now several indications of the fact that the dynamics of (\ref{Hutchinson_logistic}) is considerably more complicated than the dynamics of (\ref{chemo_logistic}) and this is not the end of the story.

In order to investigate (\ref{Hutchinson_logistic}) further, we introduce $\tau=(am-1)t$, $\rho=(am-1)r$ and
\begin{displaymath}
\xi=\frac{x}{\frac{am-1}{a}}-1.
\end{displaymath}
We get
\begin{equation}
\frac{d\xi(\tau)}{d\tau}=-\xi(\tau-\rho)(1+\xi(t))
\label{wrights_eq}
\end{equation}
which is the famous Wright's (1955)\nocite{wright.JfRAM:194} equation. The following is known about (\ref{wrights_eq}). First, its solutions are bounded (Proposition 5.13 in Smith (2011)\nocite{smith_delays}). For $0\leq \rho\leq 1.5706$, every orbit tends to the survival state (Wright (1955)\nocite{wright.JfRAM:194} and B{\'{a}}nheley, Czendes, Kristin, and Neymaier (2014)\nocite{Banhelyi.SIAMJADS:29}). Note that $1.5706<\frac{\pi}{2}\approx 1.5707963268$ and that Wright's (1955)\nocite{wright.JfRAM:194} conjecture claims this to be true for all values of the delay in $0\leq \rho\leq\frac{\pi}{2}$. This problem has been open for almost 60 years now and other open problems connected with Wright's equation exist, too, cf B{\'{a}}nheley, Czendes, Kristin, and Neymaier (2014)\nocite{Banhelyi.SIAMJADS:29}) and references therein. We can therefore ask whether it is really a simplification to use logistic equations instead of the chemostat when modeling the growth rate of the lowest trophic level. The single-species delayed chemostat has a complete qualitative analysis. There exists at least one periodic solution for (\ref{wrights_eq}) when $\rho>\frac{\pi}{2}$.

\section{Summary}

In this paper we have studied a delayed model for the single species chemostat. The global dynamics of this model has been completely known since Ellermeyer (1994)\nocite{SIAMJAM.Ellermeyer} and new variants of such an analysis has been given after this (Smith(2011)\nocite{smith_delays}). All solutions tend to either the washout or the survival state and the local stability properties of those equilibria fully determines the global properties of the model.

We make a new interpretation of these results here. The results can be divided into motion towards a hyperplane in $C([-r,0],{\bf{R}}^2)$ and motion in the hyperplane governed by a scalar monotone delay differential equation. We name this equation, the hyperbolic model. Such a separation is not always valid even in the finite dimensional case (Thieme (1992)\nocite{JoMB.Thieme:30}). We verify that no qualitative differences occur between the chemostat equations and the hyperbolic model.

A limiting case of this hyperbolic model correspond to a delayed logistic model. We prove that the delayed logistic growth rate model encountered in this way does not correspond to the model that usually is referred to as the delayed logistic model in the literature (Hutchinson (1948)\nocite{Hutchinson.Ann_NYSci:50}). This model does not satisfy any monotonicity conditions (Smith (1995)\nocite{Smith.monotone}) and is equipped with a number of problems regarding its long-run dynamical behavior that has been open for a long time, see B{\'{a}}nheley, Czendes, Kristin, and Neymaier (2014)\nocite{Banhelyi.SIAMJADS:29}.

\bibliographystyle{abbrv}
\bibliography{artiklar,biologi}
\end{document}